\documentclass[12pt]{amsart}
\usepackage{amsfonts, amssymb, latexsym, array}
\usepackage{fullpage,color}
\usepackage{verbatim,euscript}

\numberwithin{equation}{section}

\usepackage[colorlinks=true,linkcolor=blue,urlcolor=violet,citecolor=magenta]{hyperref}

\input {cyracc.def}
\font\tencyr=wncyr10 %scaled \magstephalf
\font\tencyi=wncyi10 %scaled \magstephalf
\font\tencysc=wncysc10 %scaled \magstephalf
\def\rus{\tencyr\cyracc}
\def\rusi{\tencyi\cyracc}
\def\rusc{\tencysc\cyracc}

\newtheorem{thm}{Theorem}[section]
\newtheorem{prop}[thm]{Proposition}%[chapter]

\theoremstyle{remark}
\newtheorem{rmk}[thm]{Remark}

\theoremstyle{definition}
\newtheorem{ex}[thm]{Example}
\newtheorem{df}{Definition}

\newcommand {\ah}{{\mathfrak a}}

\newcommand {\g}{{\mathfrak g}}
\newcommand {\h}{{\mathfrak h}}

\newcommand {\es}{{\mathfrak s}}
\newcommand {\te}{{\mathfrak t}}

\newcommand {\slv}{{\mathfrak{sl}(V)}}

\newcommand {\BQ}{{\mathbb Q}}

\newcommand {\sfr}{\mathsf R}

\newcommand{\lb}{\lambda}
\newcommand{\ap}{\alpha}

\newcommand{\eus}{\EuScript}

\newcommand {\ad}{{\mathrm{ad}}}

\newcommand {\hot}{{\mathsf{ht}}}
\newcommand {\ind}{{\mathsf{ind}}}

\newcommand {\tr}{{\mathrm{tr\,}}}

\newcommand {\tri}{\mathfrak{sl}_2}

\newcommand {\GR}[2]{{\textrm{{\bf #1}}}_{#2}}

\newcommand {\un}{\underline}

\newfam\eusfam%
\font\Bbbfont=msbm10 scaled 1200%
\def\bbk{\hbox {\Bbbfont\char'174}}

\begin{document}
\setlength{\parskip}{3pt plus 5pt minus 0pt}
\hfill { {\scriptsize August 14, 2008}}
\vskip1ex

\title%
{On the Dynkin index of a principal $\tri$-subalgebra}
\author[D.\,Panyushev]{Dmitri I.~Panyushev}
\address[]{Independent University of Moscow,
Bol'shoi Vlasevskii per. 11, 119002 Moscow, \ Russia
\hfil\break\indent
Institute for Information Transmission Problems, B. Karetnyi per. 19, Moscow 127994
}
\email{panyush@mccme.ru}
\urladdr{http://www.mccme.ru/~panyush}
\thanks{Supported in part by  R.F.B.R. grant 
06--01--72550.}
\maketitle

\section*{Introduction}
\noindent
The ground field $\bbk$ is algebraically closed and of characteristic zero.
Let $\g$ be a simple Lie algebra  
over $\bbk$. 
The goal of this note is to prove a closed formula for the Dynkin index of a 
principal $\tri$-subalgebra of $\g$, see Theorem~\ref{thm:main}.
The key step in the proof  uses  the ``strange formula'' of Freudenthal--de Vries.
As an application, we (1) compute the Dynkin index any simple $\g$-module regarded as $\tri$-module
and (2) obtain an identity connecting the exponents of $\g$ and the dual Coxeter numbers of 
both $\g$ and $\g^\vee$, see Section~\ref{sect:appl}.

%%%%%%%%%%%%%%%%%%%%%
\section{The Dynkin index of representations and subalgebras} 
\label{sect:di}
%%%%%%%%%%%%%%%%%%%%%

\noindent
Let $\g$ be a simple finite-dimensional Lie algebra of rank $n$. Let $\te$ be a Cartan
subalgebra, and $\Delta$ the set of roots of $\te$ in $\g$. Choose a set of positive roots
$\Delta^+$ in $\Delta$. Let $\Pi$ be the set of simple roots and $\theta$ the highest root in
$\Delta^+$. As usual, $\rho=\frac{1}{2}\sum_{\gamma>0}\gamma$.
The $\BQ$-span of all roots is a ($\BQ$-)subspace of $\te^*$, denoted $\eus E$.
Choose a non-degenerate invariant symmetric bilinear form
$(\ ,\ )_\g$ on $\g$ %normalised 
as follows. The restriction of $(\ ,\ )_\g$ to $\te$ is 
non-degenerate, hence it induces the isomorphism of $\te$ and $\te^*$ and 
a non-degenerate bilinear form on $\te^*$. We require that $(\theta,\theta)_\g=2$,
i.e.,  $(\beta,\beta)_\g=2$ of any {long} root $\beta$ in $\Delta$.

\begin{df}[E.B.~Dynkin]  {\ }\phantom{\ }
\begin{enumerate}
\item Let  $\es$ be a simple subalgebra of  $\g$.
The {\it Dynkin index\/} of $\es$ in $\g$ %, denoted $\ind(\es\hookrightarrow\g)$,
is defined by  
\[
   \ind(\es\hookrightarrow\g)=\displaystyle \frac{(x,x)_\g}{(x,x)_\es}, \quad  x\in\es . 
\]
\item If $\nu: \g\to \slv$  is a representation of $\g$, then the {\it Dynkin index of
the representation\/}, denoted $\ind_D(\g,V)$ or $\ind_D(\g,\nu)$, is defined by
\[
  \ind_D(\g,V)=\ind(\g\hookrightarrow\slv) .
\]
\end{enumerate}
\end{df}

\noindent
It is not hard to verify that, for the simple Lie algebra $\slv$, the normalised bilinear form
is given by  $(x,x)_{\slv}= \tr(x^2)$, $x\in\slv$.
Therefore, a more explicit expression for the Dynkin index of a representation 
$\nu: \g\to \slv$ is
\begin{equation}   \label{eq:ind_yavno}
  \ind_D(\g,V)=\frac{\tr \bigl(\nu(x)^2\bigr)}{(x,x)_\g} .
\end{equation}
Conversely,  the index of a simple subalgebra can 
be expressed via indices of  representations. Namely, 
\begin{equation}    \label{eq:ind-subalg}
   \ind(\es\hookrightarrow\g)=\frac{\ind_D(\es, \g)}{\ind_D(\g,\ad_\g)} .
\end{equation}
The denominator in the right hand side represents the index of the adjoint representation
of $\g$, and the numerator represents the index of the $\es$-module $\g$.

The following properties easily follow from the definition:
\\
{\sf Multiplicativity}:
 If $\h\subset\es\subset\g$ are simple Lie algebras, then 
$\ind(\h\subset\es){\cdot}\ind(\es\subset\g)=\ind(\h\subset\g)$.
\\
{\sf Additivity}:  \ $\ind_D(\g, V_1\oplus V_2)=\ind_D(\g, V_1)+\ind_D(\g, V_2)$.
It is therefore sufficient to determine the indices for the irreducible representations.

\begin{thm}[Dynkin, \protect{\cite[Theorem\,2.5]{dy}}]   \label{thm:d-1952}
Let $V_\lb$ be a  simple finite-dimensional $\g$-module with highest weight $\lb$. 
Then 
\[
   \ind_D(\g, V_\lb)=\frac{\dim V_\lb}{\dim\g} (\lb,\lb+2\rho)_\g .
\]
\end{thm}

\noindent
Although it is not obvious from the definition, the Dynkin index of a representation is an integer.
This was  proved by E.B.~Dynkin \cite[Theorem\,2.2]{dy} using lengthy classification results.
Later, he gave a better proof that is based on  a topological interpretation of the index.
A short algebraic proof is given in \cite[Ch.\,I, \S 3.10]{on}.

\begin{ex}   \label{ex:1}
{\ }\phantom{\ }\\
1) \ Let $\sfr_d$ be the simple $\tri$-module of dimension $d+1$. Then
$\ind_D(\tri,\sfr_d)= \genfrac{(}{)}{0pt}{}{d+2}{3}$.

\noindent
2) \ Recall that $\theta$ is the highset root in $\Delta^+$.
By Theorem~\ref{thm:d-1952}, 
\[
  \ind_D(\g,\ad)=(\theta,\theta+2\rho)_\g=
(\theta,\theta)_\g(1+(\rho,\theta^\vee)_\g)=2(1+(\rho,\theta^\vee)_\g) .
\]
Note that the value $(\rho,\theta^\vee)_\g$ does not depend on the normalisation of the bilinear form.
The integer $1+(\rho,\theta^\vee)$ is customary called the {\it dual Coxeter number\/}
of $\g$, and we denote it by $h^*(\g)$. Thus, $\ind_D(\g,\ad)=2h^*(\g)$.
 In the simply-laced case, $h^*(\g)=h(\g)$---the usual
Coxeter number. For the other simple Lie algebras, we have
$h^*(\GR{B}{n})=2n{-}1$, $h^*(\GR{C}{n})=n{+}1$, $h^*(\GR{F}{4})=9$,
$h^*(\GR{G}{2})=4$.
\end{ex}
Andreev, Vinberg, and Elashvili applied the Dynkin index of representations to some
invariant-theoretic problem \cite{AVE}. To this end, they adjusted the index so that it does not
depend on the choice of a bilinear form on $\g$.

\begin{df}[Andreev--Vinberg--Elashvili, 1967]  \label{def:ave}
Let $\nu: \g\to \slv$ be a finite-dimensional representation of a simple Lie algebra. Then 
\[
   \ind_{AVE}(\g,V):=\frac{\ind_D(\g, V)}{\ind_D(\g,\ad)}=
   \frac{\tr \bigl(\nu(x)^2\bigr)}{\tr\bigl(\ad_\g(x)^2
      \bigr)}, \quad x\in\g .
\]
\end{df}

\noindent
It follows that 
$\ind_{AVE}(\g,\ad_\g)=1$
\ and \ %$\displaystyle \ind_{AVE}(\g,V$. Therefore
\[
   \ind_{AVE}(\g,V_\lb)=\frac{\dim V_\lb}{\dim\g}\cdot 
               \frac{(\lb,\lb+2\rho)_\g}{(\theta, \theta+2\rho)_\g} .
\]

%%%%%%%%%%%%%%%%%%%%%%
\section{The ``strange formula''} 
\label{sect:strange}
%%%%%%%%%%%%%%%%%%%%%%

\noindent 
Let $\eus K$ be the Killing form on $\g$, i.e., $\eus K(x,x)=\tr(\ad_\g(x)^2)$, $x\in \g$.
The induced bilinear form on $\te^*$ (and $\eus E$) is denoted by $\langle\ ,\ \rangle$. 
It is the so-called
{\it canonical\/} bilinear form on $\eus E$. The canonical bilinear form is characterised
by the following property:
\begin{equation}   \label{eq:canon-2}
    \langle v,v\rangle=\sum_{\gamma\in\Delta}\langle v,\gamma\rangle \langle v,\gamma\rangle
    =2\sum_{\gamma>0}\langle v,\gamma\rangle \langle v,\gamma\rangle \ \text{ for any } v\in\eus E  .
\end{equation}
The ``strange formula''  of Freudenthal--de Vries (see \cite[47.11]{FdV}) is  
\[
     \langle \rho,\rho\rangle= \frac{\dim\g}{24} .
\]
Using our normalisation of $(\ ,\ )_\g$, the ``strange formula'' reads
\begin{equation}   \label{eq:strange}
    (\rho,\rho)_\g=\frac{\dim\g}{12}h^*(\g) .
\end{equation}
Indeed, it is well known that $\langle\theta,\theta\rangle=1/h^*(\g)$
(see e.g. \cite[Lemma\,1.1]{Cas}). Therefore, the transition factor
between two forms 
$\langle\ ,\ \rangle$ and $(\ ,\ )_\g$ (considered as forms on $\eus E$) equals $2h^*(\g)$.
Using the transition factor, we can also rewrite Eq.~\eqref{eq:canon-2} in terms of $(\ ,\ )_\g$
:
\begin{equation}   \label{eq:norm-2}
   h^*(\g)( v,v)_\g=\sum_{\gamma>0}(v,\gamma)_\g (v,\gamma)_\g .
\end{equation}

%%%%%%%%%%%%%%%%%%%%%%%%%%%
\section{The index of a principal $\tri$-subalgebra} 
\label{sect:algebra}
%%%%%%%%%%%%%%%%%%%%%%%%%%%

\noindent 
If $e\in\g$ is nilpotent, then the exists a subalgebra $\ah\subset\g$ such that $\ah\simeq\tri$
and $e\in\ah$ (Morozov, Jacobson). If $e$ is a {\it principal\/} nilpotent element, then the corresponding  $\tri$-subalgebra is also called principal. (See \cite[\S\,9]{dy} and 
\cite[Sect.\,5]{ko59} for properties of
principal $\tri$-subalgebras.)
Let $(\tri)^{pr}$ be a principal $\tri$-subalgebra of $\g$.
In this section, we obtain a uniform expression for 
$\ind((\tri)^{pr}\hookrightarrow \g)$. 

Recall that $\Delta$ has at most two root lengths. Let $\theta_s$ denote the short dominant
root in $\Delta^+$. (Hence $\theta=\theta_s$ if and only if $\Delta$ is simply-laced.)
Set $r=\|\theta\|^2/\|\theta_s\|^2 \in\{1,2,3\}$.
Along with $\g$, we also consider the Langlands
dual  algebra $\g^\vee$, which is determined by the dual root system $\Delta^\vee$.
Since the Weyl groups of $\g$ and $\g^\vee$ are isomorphic, we have $h(\g)=h(\g^\vee)$.
%and hence $\dim\g=\dim\g^\vee$. 
However, the dual Coxeter numbers can be different (cf. $\GR{B}{n}$ and $\GR{C}{n}$). 

The half-sum of positive roots for $\g^\vee$ is
\[
   \rho^\vee:=\frac{1}{2}\sum_{\gamma>0} \gamma^\vee=\sum_{\gamma>0}\frac{\gamma}{
(\gamma,\gamma)_\g} .
\] 
It is well-known (and easily verified) that $(\rho^\vee,\gamma)_\g=\hot(\gamma)$ for any 
$\gamma\in\Delta^+$. (This equality does not depend on the normalisation of a bilinear form.)
It follows that $h^*(\g^\vee)=(\rho^\vee, \theta_s)=\hot(\theta_s)$.

\begin{prop}  \label{prop:ht2}
For a simple Lie algebra $\g$ with the corresponding root system $\Delta$, we have
\begin{equation}   \label{eq:ht2}
   \sum_{\gamma>0} \hot^2(\gamma)=\frac{\dim\g}{12}h^*(\g)h^*(\g^\vee)r .
\end{equation}
\end{prop}\begin{proof}
The equality in \eqref{eq:ht2} is essentially equivalent to the "strange formula". 

Applying Eq.~\eqref{eq:norm-2}  to  $v=\rho^\vee$, we obtain
\begin{equation}   \label{eq:hot2}
  h^*(\g)(\rho^\vee,\rho^\vee)_\g=\sum_{\gamma>0}(\rho^\vee,\gamma)_\g(\rho^\vee,\gamma)_\g=\sum_{\gamma>0}\hot^2(\gamma) .
\end{equation}
For $\g^\vee$, the strange formula says that
$(\rho^\vee,\rho^\vee)_{\g^\vee}=\displaystyle\frac{\dim\g}{12}h^*(\g^\vee)$.
Although the normalised bilinear forms $(\ ,\ )_\g$ and  $(\ ,\ )_{\g^\vee}$ are proportional
upon restriction to $\eus E$, they are not 
equal in general. 
Indeed,  the square of the length of a long root in $\Delta^{\vee}$ with respect to 
$(\ ,\ )_\g$ equals $2r$. Hence the transition factor is $r$ and 
\begin{equation}   \label{eq:ro-check}
    (\rho^\vee,\rho^\vee)_\g=r(\rho^\vee,\rho^\vee)_{\g^\vee}=\frac{\dim\g}{12}h^*(\g^\vee)r .
\end{equation}
Then the  assertion follows from \eqref{eq:hot2} and \eqref{eq:ro-check}.
\end{proof}

\begin{thm}   \label{thm:main}
$\ind((\tri)^{pr}\hookrightarrow \g)=\displaystyle\frac{\dim\g}{6}h^*(\g^\vee) r$.
% where $r=(\theta,\theta)/(\theta_s,\theta_s)$.
\end{thm}\begin{proof}
Combining Eq.~\eqref{eq:ind-subalg}, Example~\ref{ex:1}(2), and Definition~\ref{def:ave} yields the
following  formula for the index of a simple subalgebra $\es$ in $\g$:
\begin{equation}  \label{eq:ind-s-g}
    \ind(\es\hookrightarrow\g)=\frac{h^*(\es)}{h^*(\g)}\cdot \ind_{AVE}(\es, \g) .
\end{equation}
We use this formula with $\es=(\tri)^{pr}$. 
Let $h$ be the semisimple element of a principal $\tri$-triple. Without loss of generality, we may  assume that $h$ is dominant. Then $\ap(h)=2$ for any $\ap\in\Pi$. 
Put $\tilde h=h/2$. Then $\gamma(\tilde h)=\hot(\gamma)$ for any $\gamma\in\Delta$ and
$\ad\,\tilde h$ has the eigenvalues $-1,\,0,\,1$ in  $(\tri)^{pr}$. Hence
\[
   \ind_{AVE}((\tri)^{pr}, \g)=\frac{\tr(\ad_\g\, \tilde h)^2}{\tr(\ad_\es\,\tilde h)^2}=
   \frac{\sum_{\gamma\in\Delta}\hot^2(\gamma)}{2}=\sum_{\gamma>0}\hot^2(\gamma) .
\]
Since $h^*(\tri)=2$, the theorem follows from Proposition~\ref{prop:ht2} and 
Eq.~\eqref{eq:ind-s-g}.
\end{proof}

Below, we tabulate the values of index for all simple Lie algebras.

\begin{center}
\begin{tabular}{c|ccccccccc|}
$\g$ & $\GR{A}{n}$ & $\GR{B}{n}$ & $\GR{C}{n}$& $\GR{D}{n}$ & $\GR{E}{6}$ & $\GR{E}{7}$ &  $\GR{E}{8}$ & $\GR{F}{4}$ & $\GR{G}{2}$ \\ \hline 
$\ind((\tri)^{pr}\hookrightarrow \g)$ & \rule{0pt}{2.7ex}$\genfrac{(}{)}{0pt}{}{n+2}{3}$ &
$\frac{n(n+1)(2n+1)}{3}$ & $\genfrac{(}{)}{0pt}{}{2n+1}{3}$ & $\frac{(n-1)n(2n-1)}{3}$ &
$156$ & $399$ & $1240$ & $156$ & $28$ \\
 \end{tabular}
\end{center}

\begin{rmk}
For the exceptional Lie algebras,  Dynkin computed the indices of \un{all} $\tri$-subalgebras,
see \cite[Tables\, 16--20]{dy}.
\end{rmk}

Note that the index of a principal $\tri$ is preserved under the unfolding procedure
$\g\leadsto \tilde\g$
applied to multiply laced Dynkin diagram. Namely, 
$\ind((\tri)^{pr}\hookrightarrow \g)=\ind((\tri)^{pr}\hookrightarrow \tilde\g)$, where the four pairs
$(\g,\tilde\g)$ are: $(\GR{C}{n},\GR{A}{2n-1})$,  $(\GR{B}{n},\GR{D}{n+1})$, 
$(\GR{F}{4},\GR{E}{6})$, $(\GR{G}{2},\GR{D}{4})$.
This is, of course, explained by the multiplicativity of the index of subalgebras and the fact that $\ind(\g\hookrightarrow \tilde\g)=1$.

\begin{rmk}
Proposition~\ref{prop:ht2} provides a uniform expression for $\sum_{\gamma>0}
\hot^2(\gamma)$. One might ask for a similar formula for $\sum_{\gamma>0}\hot(\gamma)$.
However, such a formula seems to only exist in the simply-laced case. Indeed, for any
$\g$ we have $2(\rho,\rho^\vee)_\g=\sum_{\gamma> 0}(\gamma,\rho^\vee)_\g=
\sum_{\gamma>0}\hot(\gamma)$. If $\Delta$ is simply-laced, then $\rho^\vee=
2\rho/(\theta,\theta)_\g=\rho$, and using the ``strange formula'' one obtains
\[
  \sum_{\gamma>0}\hot(\gamma)=2(\rho,\rho)_\g=\frac{\dim\g}{6}h(\g) \ .
\]
{\sl Question.} Consider the function $s\mapsto f(s)=\sum_{\gamma>0}\hot^s(\gamma)$.
Are there some other values of $s$ such that $f(s)$ has a nice closed expression ?
\end{rmk}

%%%%%%%%%%%%%%%%%%%%%%
\section{Some applications} 
\label{sect:appl}
%%%%%%%%%%%%%%%%%%%%%%

\noindent
{\bf (A)} \ Let $\nu:\g\to \mathfrak{sl}(V_\lb)$ be an irreducible representation.  
Our first observation is that using Theorems~\ref{thm:d-1952} and \ref{thm:main} we can immediately
compute the Dynkin index of $V_\lb$ as $(\tri)^{pr}$-module:
\begin{multline*}
\ind_D((\tri)^{pr}, V_\lb)=\ind((\tri)^{pr}\hookrightarrow \mathfrak{sl}(V_\lb))=
\ind((\tri)^{pr}\hookrightarrow \g){\cdot}\ind(\g\hookrightarrow  \mathfrak{sl}(V_\lb))= \\
  \ind((\tri)^{pr}\hookrightarrow \g) \cdot \ind_D(\g,V_\lb)=
  \frac{\dim\g}{6} h^*(\g^\vee)r \cdot \frac{\dim V_\lb}{\dim\g}(\lb,\lb+2\rho)_\g=
  \frac{\dim V_\lb}{6}{\cdot} h^*(\g^\vee){\cdot}r{\cdot} (\lb,\lb+2\rho)_\g \ .
\end{multline*}
Furthermore, we have
\begin{equation}     \label{eq:D-AVE}
   \ind_D((\tri)^{pr}, V_\lb)=\ind_D(\tri, \ad){\cdot}\ind_{AVE}((\tri)^{pr}, V_\lb)=
   4{\cdot}\ind_{AVE}((\tri)^{pr}, V_\lb) 
\end{equation}
and
\[
   \ind_{AVE}((\tri)^{pr}, V_\lb)=\frac{\tr(\nu(\tilde h)^2)}{\tr((\ad\, \tilde h)^2)}=
   \frac{\sum_{\mu\dashv  V_\lb} \mu(\tilde h)^2}{2} .
\]
where notation $\mu\dashv V_\lb$ means that $\mu$ is a weight of $V_\lb$, and the
sum runs over all weights according to their multiplicities. 
Since $\mu(\tilde h)=(\mu,\rho^\vee)_\g$, we finally obtain
\begin{equation}  \label{eq:vee-kvadrat}
   \sum_{\mu\dashv  V_\lb}(\mu,\rho^\vee)^2_\g=
   \frac{\dim V_\lb}{12} {\cdot}h^*(\g^\vee){\cdot}r{\cdot} (\lb,\lb+2\rho)_\g  \ .
\end{equation}   
This can be compared with the formula of Freudenthal--de Vries
(see \cite[47.10.2]{FdV}):
\begin{equation}  \label{eq:rho-kvadrat}
  \sum_{\mu\dashv  V_\lb}\langle\mu,\rho\rangle^2=\frac{\dim V_\lb}{24}
  \langle \lb,\lb+2\rho\rangle  \ .
\end{equation}
One can verify that Eq.~\eqref{eq:vee-kvadrat} and \eqref{eq:rho-kvadrat}
agree in the simply-laced case, where $\rho$ is proportional to $\rho^\vee$.

{\bf (B)} \ Let $m_1,\dots, m_n$ be the exponents of $\g$.  Regarding $\g$ as $(\tri)^{pr}$-module,
one has $\g= \bigoplus_{i=1}^n \sfr_{2m_i}$ \cite[Cor.\,8.7]{ko59}. 
Then using Example~\ref{ex:1}(1),
Eq.~\eqref{eq:ind-s-g}, \eqref{eq:D-AVE}, and the additivity of the index of representations, 
we obtain the identity
\begin{multline*}
 \frac{\dim\g}{6} h^*(\g^\vee)r  =\ind((\tri)^{pr}\hookrightarrow \g)=
 \frac{h^*(\tri)}{h^*(\g)} \sum_{i=1}^n \ind_{AVE}(\tri,\sfr_{2m_i})=
  \\
  \frac{1}{2h^*(\g)}\sum_{i=1}^n\ind_D(\tri,\sfr_{2m_i})=\frac{1}{2h^*(\g)}\sum_{i=1}^n
  \genfrac{(}{)}{0pt}{}{2m_i+2}{3} .
\end{multline*}

\end{document}